\documentclass[12pt,reqno]{amsart}

\usepackage[usenames, dvipsnames]{color}
\usepackage{amsmath, amsthm, amssymb}
\usepackage{amsfonts}
\usepackage[ansinew]{inputenc}
\usepackage[dvips]{epsfig}
\usepackage{graphicx}
\usepackage[english]{babel}
\usepackage{enumerate}
\usepackage{hyperref}
\theoremstyle{plain}
\newtheorem{thm}{Theorem}[section]

\newtheorem{prop}[thm]{Proposition}

\theoremstyle{definition}

\theoremstyle{remark}
\newtheorem{rem}[thm]{Remark}

\numberwithin{equation}{section}

\newcommand{\R}{\mathbb{R}}

\newcommand{\average}{{\mathchoice {\kern1ex\vcenter{\hrule height.4pt
width 6pt depth0pt} \kern-9.7pt} {\kern1ex\vcenter{\hrule
height.4pt width 4.3pt depth0pt} \kern-7pt} {} {} }}

\def\R{\mathbb{R}}

\begin{document}

\title[The two membranes problem for fully nonlinear operators]{The two membranes problem for fully nonlinear operators}

\author{Luis Caffarelli}

\address{University of Texas at Austin, Department of Mathematics, 2515 Speedway, TX 78712 Austin, USA}
\email{caffarel@math.utexas.edu}

\author{Luis Duque}

\address{University of Texas at Austin, Department of Mathematics, 2515 Speedway, TX 78712 Austin, USA}

\email{lduque@math.utexas.edu}

\author{Hern\'an Vivas}

\address{Departamento de Matem\'atica, FCEyN, Universidad Nacional de Mar del Plata, Conicet, Dean Funes 3350, 7600 Mar del Plata, Argentina}
\email{hvivas@gmail.com}

\keywords{Free boundary problems, fully nonlinear}

\begin{abstract}

We study the two membranes problem for two different fully nonlinear operators. We give a viscosity formulation for the problem and prove existence of solutions. Then we 
prove a general regularity result and the optimal $C^{1,1}$ regularity when the operators are the Pucci extremal operators. We also give an example that shows that no 
regularity for the free boundary is to be expected to hold in general. 

\end{abstract}

\maketitle

\section{Introduction}
\label{sec.1}

The two membranes problem was first studied by Vergara-Caffarelli \cite{VC} in the context of variational inequalities to describe the equilibrium position of two elastic membranes in contact with each other that are not allowed to cross. He considered the 
linear elliptic case, in which the problem can be reduced to the classical obstacle problem by looking at the difference between the two functions representing the position of each membrane.  

Nearly 35 years later, Silvestre \cite{S} studied the problem for a nonlinear operator in divergence form. He obtained the optimal $C^{1,1}$ regularity of solutions together with a characterization of the regularity of the free boundary, 
that is the boundary of the set where the two functions coincide. The strategy in his proof was to show that the difference of the two functions satisfies an obstacle problem for the linearized operator, for which the regularity theory of 
the solutions and the free boundary are well known. An important remark is that in both of these cases the operator governing the behavior of each function is the same. 

In a recent paper, Caffarelli, De Silva and Savin \cite{CDS} considered the two membranes problem for (possibly nonlocal) different operators, i.e. they consider the case in which one of the membranes (say the lower one) satisfies an equation that has 
higher order with respect to the other one. Here, heuristically, the lower order operator can be treated as a perturbation and some regularity for the lower membrane is obtained. Regularity from the upper membrane can then be deduced by solving an obstacle problem
(with the lower membrane as obstacle) and obtaining estimates for solutions of nonlocal obstacle type problems in which the obstacle is not smooth. Repeating these arguments, the optimal regularity is achieved. 

We also point out that the problem has been studied by several authors in the general case of $N$ membranes, see \cite{CV}, \cite{CCVC}, \cite{ARS}. 

Here, motivated by a model from mathematical finance, we consider a version of the two membrane problem for two different fully nonlinear operators. It is worth pointing out that,
for the case of two different operators of the same order, the only result available (to the best of the authors' knowledge) is the H\"older regularity obtained in \cite{CDS} (see the Introduction
there for a discussion of the difficulties of this problem). In this paper we prove $C^{1,\alpha}$ regularity of the solution pair for (concave or convex) operators satisfying a sort of 
compatibility condition (see \eqref{eq.comp} below) and $C^{1,1}$ regularity for the case of the Pucci extremal operators, which is optimal. Moreover, we give an explicit example 
that shows that no regularity can be expected to hold for the free boundary in general.

\subsection{Notation and preliminaries}

Throughout this paper the ellipticity constants $\lambda,\Lambda\in\R$ will be fixed and will satisfy $0<\lambda<\Lambda$. Given these, we denote by $\mathcal{M}^+$ and $\mathcal{M}^-$ the Pucci extremal operators with respect to the class of 
symmetric matrices whose eigenvalues lie between $\lambda$ and $\Lambda$, that is for any symmetric matrix $X$

\begin{equation*}
 \mathcal{M}^+(X)=\sup_{A\in\mathcal{L}_{\lambda,\Lambda}}\textrm{tr}(AX)\quad\quad\textrm{and}\quad\quad\mathcal{M}^-(X)=\inf_{A\in\mathcal{L}_{\lambda,\Lambda}}\textrm{tr}(AX)
\end{equation*}
where
\[
 \mathcal{L}_{\lambda,\Lambda}=\{A\in\R^{n\times n}:A\textrm{ is symmetric and } \lambda Id\leq A \leq\Lambda Id\}
\]
$X\geq Y$ meaning as usual that $X-Y$ is a positive semidefinite matrix. 

Recall that an operator $F:\R^{n\times n}\rightarrow\R$ is said to be uniformly elliptic with repect to the class $\mathcal{L}_{\lambda,\Lambda}$ if it satisfies
\begin{equation}\label{eq.ellip}
 \mathcal{M}^-(X-Y)\leq F(X)-F(Y)\leq  \mathcal{M}^+(X-Y)
\end{equation}
for any pair of symmetric matrices $X$ and $Y$. 

We will assume without loss of generality that $F(0)=0$. A useful remark that follows from \eqref{eq.ellip} is that if $u$ is a function satisfying $F(D^2u)=f$ then 
\begin{equation} \label{eq.diri}
  \left\{ \begin{array}{rcl}
  \mathcal{M}^+(D^2u) & \geq & f \\
  \mathcal{M}^-(D^2u) & \leq & f.
\end{array}\right.
\end{equation}
In particular, $u$ is a subsolution and a supersolution of two (possibly different) elliptic equations with bounded measurable coefficients. 

\subsection{The two membranes problem for fully nonlinear operators}

The problem we will consider is the following: given two functions $u_0,v_0\in C^\gamma(\partial B_1)$ and $f,g\in C^\gamma(B_1)$ for some $\gamma\in(0,1)$, we want to study the solutions $u$ and $v$ of

\begin{equation}\label{eq.main}
  \left\{ \begin{array}{rcll}
  u & \geq & v & \textrm{ in } B_1 \\
  F(D^2u) & \leq & f & \textrm{ in } B_1 \\
  G(D^2v) & \geq & g & \textrm{ in } B_1 \\
  F(D^2u) & = & f & \textrm{ in } B_1\cap\Omega \\
  G(D^2v) & = & g & \textrm{ in } B_1\cap\Omega \\
  u & = & u_0 & \textrm{ on } \partial B_1 \\ 
  v & = & v_0 & \textrm{ on } \partial B_1 
\end{array}\right.
\end{equation}
where
\[
\Omega:=\{u>v\},
\]
$F$ is convex and 
\begin{equation}\label{eq.comp}
G(X)=-F(-X).  
\end{equation}
Note that $G$ thus defined will be concave and that particular examples are
\begin{equation*}
F(D^2w)=\sup_{\alpha\in \Sigma}\:\textrm{tr}(A_\alpha D^2 w) \quad\quad\textrm{and}\quad\quad G(D^2w)=\inf_{\alpha\in \Sigma}\:\textrm{tr}(A_\alpha D^2 w) 
\end{equation*}
with $\Sigma$ some set of indexes and $A_\alpha\in\mathcal{L}_{\lambda,\Lambda}$ for every $\alpha$. If $A_\alpha$ can be any matrix in $\mathcal{L}_{\lambda,\Lambda}$ then $F=\mathcal{M}^+$ and $G=\mathcal{M}^-$. It is in this latter case that we prove the optimal regularity. Note that the strict inequality assumed for the ellipticity constants avoids these operators to become just a multiple of the Laplacian.  

Equation \eqref{eq.main} is to be understood in the viscosity sense. More precisely: if $\varphi$ is a $C^2$ function in $B_1$ satisfying for some $x_0\in B_1$
\begin{equation*}
 \varphi(x)\leq u(x) \quad\textrm{in}\quad B_1,\quad\quad \varphi(x_0)=u(x_0)
\end{equation*}
(i.e. $\varphi$ \emph{touches $u$ by below} at $x_0$) then  
\begin{equation*}
F(D^2\varphi(x_0))\leq f(x_0).
\end{equation*}

Similarly if $\varphi$ touches $v$ by above and of course the opposite inequalities (last two equations in \eqref{eq.main})  hold if $\varphi$ touches $u$ by above 
(or $v$ by below) in $\Omega$. A simple remark that will be useful is that it is equivalent to use paraboloids instead of general $C^2$ functions.

Note that the convexity of $F$ as well as the H\"older regularity for $f$ and $g$ are natural assumptions if we want to get optimal regularity. In fact, one expect solutions to this problem to be $C^{1,1}$ as long as the equation they solve on the noncontact
set is ``good enough'', meaning that we have at least $C^{1,1}$ regularity for it. This, in principle, is not true in general if $F$ is not convex or $f$ and $g$ are merely bounded. Also, for the problem to make sense we will assume throughout the paper that $u_0>v_0$ on $\partial B_1$. Moreover, we will assume that $f-g\geq0$. Notice that if this was not the case the problem could loose interest and degenerate into just two independent fully nonlinear equations. Indeed, if $f-g<0$, we would have (see \eqref{eq.diri})
\begin{equation*}
  \left\{ \begin{array}{rcll}
  \mathcal{M}^+(D^2(v-u)) & > & 0 & \textrm{ in } B_1 \\
   v-u & < & 0 & \textrm{ in } \partial B_1 
\end{array}\right.
\end{equation*}
and due to the maximum principle $u>v$ in $B_1$. Then there is no contact set and we just have the respective equations for $u$ and $v$. 

Equation \eqref{eq.main} models a so called ``bid and ask'' situation in which we have an asset, a seller (represented by $u$) and a buyer (represented by $v$). The price 
of the asset is random and the transaction will only take place when $u$ and $v$ ``agree on a price'', i.e. when $u=v$. Moreover, we want to model the expected earnings of $u$ and $v$, assuming
that their strategy is optimal. 

One can think of this problem as having two different (although related) features: on one hand, we have an ``obstacle type'' situation, in which $u$ tries to maximize gain 
with $v$ being an obstacle and vice versa ($v$ minimizing cost and $u$ being an obstacle), hence the constraint $u\geq v$. But perhaps more interesting is the special relation
between $u$ and $v$. Because of the ``bid and ask'' nature of the model, the Bellman type equations that govern the behavior of our solutions are closely related (recall $F(X)=-G(-X)$)  and
it is precisely this feature which opens a way to get regularity even though the operators are different. 

The main result of this paper is the following:

\begin{thm}\label{thm.main}
Let $u$ and $v$ solve \eqref{eq.main} in the viscosity sense with $F=\mathcal{M}^+$ and $G=\mathcal{M}^-$. Then $u$ and $v$ are $C^{1,1}$ in $B_{1/4}$ and  
\[
 \|D^2u\|_{L^\infty(B_{1/4})}, \|D^2v\|_{L^\infty(B_{1/4})}\leq C
\]
where $C$ depends only on $n,\lambda,\Lambda,\|f\|_{C^\gamma(B_1)},\|g\|_{C^\gamma(B_1)},\|v\|_{L^\infty(B_1)}$ and $\|u\|_{L^\infty(B_1)}$.
\end{thm}

\section{Existence}

In this section we prove existence of solutions for our problem. We use the method of penalization, i.e. we are going to consider a family of unconstrained ``penalized equations''
whose solutions are uniformly bounded in some H\"older space and hence convegent up to a subsequence. Then we are going to show that the limit of that subsequence is actually a
solution to \eqref{eq.main} (see \cite{KS}).

The penalized problem we are going to consider is the following: 
\begin{equation}\label{eq.penalepsilon}
  \left\{ \begin{array}{rcll}
  F(D^2u_\epsilon) & = & f+\beta_\epsilon(u_\epsilon-v_\epsilon)& \textrm{ in } B_1 \\
  G(D^2v_\epsilon) & = & g-\beta_\epsilon(u_\epsilon-v_\epsilon)& \textrm{ in } B_1 \\
  u_\epsilon & = & u_0  & \textrm{ in } \partial B_1\\
  v_\epsilon & = & v_0 & \textrm{ in } \partial B_1 
\end{array}\right.
\end{equation}
where for each $\epsilon >0$ we define
\begin{equation}\label{eq.betaepsilon}
\beta_{\epsilon}(t)=\beta(t/\epsilon)
\end{equation}
with $\beta:\R\rightarrow\R$ a smooth function satisfying 
\begin{equation}\label{eq.beta}
 -N\leq \beta\leq0,   \quad   
\beta'\geq0,   \quad   
\beta(t) = 0\textrm{ when }t\geq 1, \quad 
\beta(t) = -N \textrm{ when }t\leq 0, \quad 
\end{equation}
and
\begin{equation}\label{eq.ene}
N:= \|f-g\|_{L^\infty(B_1)}. 
\end{equation}

To get the existence and a priori bounds for solutions of \eqref{eq.penalepsilon} we will use a fixed point argument. Hence, we will need global regularity results for 
equations of the form 

\begin{equation} \label{eq.dirichlet}
  \left\{ \begin{array}{rcll}
  H(D^2u) & = & h & \textrm{ in } B_1 \\
  u & = & u_0  & \textrm{ in } \partial B_1 \\
\end{array}\right.
\end{equation}
where $H$ is a uniformly elliptic operator. Here we mostly follow Chapter 4 of \cite{CC} but since the proofs need to be modified slightly we sketch them below for completeness:

\begin{prop} \label{prop.growth}
Let $u$ be a viscosity solution of \eqref{eq.dirichlet} with $h\in L^\infty(B_1)$ and $u_0\in C^\gamma(\partial B_1)$. Then for any $x_0\in\partial B_1$ we have 
\begin{equation}\label{eq.growth}
\sup_{x\in B_1}\frac{|u(x)-u(x_0)|}{\vert x-x_0\vert^{\gamma/2}}\leq C 
\end{equation}
where $C$ is a constant depending only on $n,\lambda,\Lambda,\gamma,\|u_0\|_{C^\gamma(\partial B_1)}$ and $\|h\|_{L^\infty(B_1)}$.
\end{prop}

\begin{proof}
We translate, rotate and add a constant to $u$ so that it is defined on $B:=B_1(1,0,..,0)$, $x_0=(0,0,...,0)$ and $u(0)=0$. We want to show 
\begin{equation}\label{eq.growth2}
\sup_{x\in B_1}\frac{|u(x)|}{\vert x\vert^{\gamma/2}}\leq C 
\end{equation}
For this let us define the barrier $\psi(x)= C x_1^{\gamma/2}$ with $C$ a constant to be determined. Notice that
\[
|x|^\gamma=(x_1^2+x_2^2+...+x_n^2)^{\frac{\gamma}{2}}=(2x_1)^{\frac{\gamma}{2}}  
\]
on $\partial B$ and hence 
\[
u(x)=u_0(x)\leq [u_0]_{C^{\gamma}(\partial B_1)}|x|^\gamma\leq C x_1^{\gamma/2}=\psi(x)
\]
there. On the other hand, $\psi$ satisfies
\[
H(D^2\psi)\leq \mathcal{M}^+(D^2 \psi)=C\lambda\frac{\gamma}{2}(\frac{\gamma}{2}-1)x_1^{\frac{\gamma}{2}-2}\leq -\Vert h \Vert_{L^\infty(B_1)}
\]
in $B$ in the viscosity sense if we take $C$ large enough. From the maximum principle it follows that $u\leq \psi$ in $B$. 

Symmetrically, we see that 
$u \geq -\psi$ on $\partial B$ and $H(D^2(-\psi))\geq \Vert h \Vert_{L^\infty(B_1)}$ in $B$, so using the maximum principle again we get $u\geq-\psi$ and hence \eqref{eq.growth2}.
\end{proof}

Now we prove the global H\"older estimates.

\begin{prop} \label{prop.globalest}
Let $u$ be a viscosity solution of \eqref{eq.dirichlet} with $h\in L^\infty(B_1)$ and $u_0\in C^\gamma(\partial B_1)$. Then 
\begin{equation}\label{eq.globalest}
\Vert u \Vert_{C^\eta(B_1)} \leq C 
\end{equation}
where $C$ is a constant depending only on $n,\lambda,\Lambda,\|u_0\|_{C^\gamma(\partial B_1)}$ and $\|h\|_{L^\infty(B_1)}$ and $\eta\leq\gamma/2$.
\end{prop}

\begin{proof}
We start by recalling that by interior estimates (Proposition 4.10 in \cite{CC}) solutions of \eqref{eq.dirichlet} are in $C^\alpha_{loc}(B_1)$ for some $\alpha>0$. Let $\eta=\min\{\gamma/2,\alpha\}$, $x_1, x_2\in B_1$
$r=\vert x_1 - x_2\vert$, and take $x_1',x_2'\in\partial B_1$ such that 
\[
d_1:= d(x_1, \partial B_1)= \vert x_1 - x_1'\vert\quad\textrm{ and }\quad d_2:= d(x_2, \partial B_1)=\vert x_2 - x_2'\vert.
\]
We assume without loss of generality that $d_2\leq d_1$ and we want to show that 
\begin{equation}\label{eq.holder}
\vert u(x_1)-u(x_2)\vert \leq C\vert x_1-x_2 \vert^\eta.
\end{equation}

We split the proof in two cases:

\textbf{Case 1:} When $r \leq \frac{d_1}{2}$, from the rescaled version of the interior estimates applied to $u-u(x_1')$ on $B_{d_1}(x_1)$ and our estimates at the boundary \eqref{eq.growth} we have

\begin{equation*} 
\begin{split}
d_1^\alpha \Vert u-u(x_1') \Vert_{C^\alpha(B_{d_1/2}(x_1))} & \leq C ( d_1^2 \Vert h \Vert_{L^\infty(B_{d_1}(x_1))}+ \Vert u-u(x_1') \Vert_{L^\infty(B_{d_1}(x_1))}) \\
 & \leq C  d_1^2\Vert h \Vert_{L^\infty(B_{d_1})}  + Cd_1^{\gamma/2} \leq C d_1^{\gamma/2}.
\end{split}
\end{equation*}
On the other hand
\[
d_1^\eta \frac{|u(x_1)-u(x_2)|}{|x_1-x_2|^\eta}\leq d_1^\alpha \frac{|u(x_1)-u(x_2)|}{|x_1-x_2|^\alpha}\leq d_1^\alpha\Vert u(\cdot)-u(x_1')\Vert_{C^\alpha(B_{d_1/2}(x_1))} 
\]
so
\[
\frac{|u(x_1)-u(x_2)|}{|x_1-x_2|^\eta}\leq C d_1^{\gamma/2-\eta}\leq C
\]
as desired in this case.

\textbf{Case 2:} When $r > \frac{d_1}{2}$, again by the boundary estimates \eqref{eq.growth} and the triangle inequality
\begin{equation*}
\begin{split}
\vert u(x_1)-u(x_2) \vert & \leq  |u(x_1)-u(x_1')| + |u(x_1')-u(x_2')| + |u(x_2')-u(x_2)|  \\
 & \leq Cd_1^{\gamma/2} + C|x_1'-x_2'|^\gamma + Cd_2^{\gamma/2}  \leq C(d_1^{\gamma/2} + r^\gamma+d_1^\gamma+d_2^\gamma + Cd_2^{\gamma/2}) \\
 & \leq C  r^{\gamma/2}= C\vert x_1-x_2\vert^{\gamma/2}. 
\end{split}
\end{equation*}
\end{proof}

In our next proof we are also going to use global H\"older estimates for ``equations with bounded measurable coefficients'':

\begin{prop}\label{prop.globalestbmc}
Let $u$ be a viscosity solution of  
\begin{equation} \label{eq.diri2}
  \left\{ \begin{array}{rcll}
  \mathcal{M}^+(D^2u) & \geq & -\alpha & \textrm{ in } B_1 \\
  \mathcal{M}^-(D^2u) & \geq & \ \ \alpha & \textrm{ in } B_1 \\
   u & = & u_0  & \textrm{ in } \partial B_1 \\
\end{array}\right.
\end{equation}
for $\alpha$ a positive constant and $u_0\in C^\gamma(\partial B_1)$. Then
\begin{equation}\label{eq.globalestbmc}
\Vert u \Vert_{C^\eta(B_1)} \leq C 
\end{equation}
where $C$ is a constant depending only on $n,\lambda,\Lambda,\|u_0\|_{C^\gamma(\partial B_1)}$ and $\alpha$ and $\eta\leq\gamma/2$.
\end{prop}

\begin{proof}
The proof follows exactly as that of Proposition \ref{prop.globalest}. We just point out that in order to prove the boundary estimates \eqref{eq.growth} a barrier argument for the Pucci extremal operators is used that is trivially adapted to a situation like \eqref{eq.diri2}. This is then combined with interior estimates that also hold for $\eqref{eq.diri2}$ (see \cite{CC}) to give \eqref{eq.globalestbmc}
\end{proof}

Finally, the following observation is going to be useful: it follows from the proof of Proposition \ref{prop.growth} that the dependence of the constant on $\|h\|_{L^\infty(B_1)}$ is continuous. The same is true for the interior estimates (again, see \cite{CC}) and hence for the constant in \eqref{eq.globalestbmc}.

We can now prove existence of the penalized problem:

\begin{prop}\label{prop.existpenal}
Let $\beta:\R\rightarrow\R$ be a smooth bounded function, $u_0,v_0\in C^\gamma(\partial B_1)$ and $f,g\in L^\infty(B_1)$ . There exist $u,v\in C(B_1)$ such that
\begin{equation}\label{eq.penal}
  \left\{ \begin{array}{rcll}
  F(D^2u) & = & f(x)+\beta(u-v)& \textrm{ in } B_1 \\
  G(D^2v) & = & g(x)-\beta(u-v)& \textrm{ in } B_1 \\
  u & = & u_0  & \textrm{ in } \partial B_1\\
  v & = & v_0 & \textrm{ in } \partial B_1
\end{array}\right.
\end{equation}
in the viscosity sense.
\end{prop}

\begin{proof}

We will use Schauder's Fixed Point Theorem (see \cite{GT}).  Let $\tilde{\alpha}=\eta/2$ with $\eta$ as in Proposition \ref{prop.globalest} and consider the map 
\[
T: C^\eta(B_1) \times C^\eta(B_1)  \longrightarrow  C^\eta(B_1) \times C^\eta(B_1)
\] 
defined as $T(\bar{u},\bar{v}):= (u,v)$, $u,v$ satisfying
\begin{equation}\label{eq.Tpenal}
  \left\{ \begin{array}{rcll}
  F(D^2u) & = & f(x)+\beta(\bar{u}-\bar{v})& \textrm{ in } B_1 \\
  G(D^2v) & = & g(x)-\beta(\bar{u}-\bar{v})& \textrm{ in } B_1 \\
  u & = & u_0  & \textrm{ on } \partial B_1\\
  v & = & v_0 & \textrm{ on } \partial B_1
\end{array}\right.
\end{equation}
Such $u$ and $v$ exist by Perron's method. Also, let
\begin{equation*}
X := \{(\bar{u}, \bar{v})\in C^{\tilde{\alpha}}(B_1) \times C^{\tilde{\alpha}}(B_1): \bar{u} = u_0,\bar{v}=v_0 \textrm{ on } \partial B_1, \Vert \bar{u} \Vert_{C^{\eta}(B_1)},  \Vert \bar{v} \Vert_{C^{\eta}(B_1)} \leq C \}   
\end{equation*}
with $C$ to be determined later. 

If we can show that $X$ is a compact convex set in $C^\eta(B_1) \times C^\eta(B_1)$, that $T$ is continuous in $X$ and $T(X)\subset X$ then by Schauder's Fixed Point Theorem there exists a solution to \eqref{eq.penal}. 
We divide the proof in several steps:

\textbf{Step 1: convexity and compactness of $X$.}

The convexity is trivial. As for the compactness, it is a straight forward consequence of the Arsela-Ascoli theorem since $\tilde{\alpha}<\eta$.

\textbf{Step 2: $T(X)\subset X$.}

Notice that if $\bar{u}, \bar{v} \in C^{\tilde{\alpha}}(B_1)$ and we let 
\[
h:=f+\beta(\bar{u}-\bar{v})\quad\quad\textrm{ and }\quad\quad k:=g-\beta(\bar{u}-\bar{v})
\]
we have, as $f,g, \beta$ are bounded, that $h,k \in L^\infty(B_1)$. Hence, from Proposition \ref{prop.globalest} we know that 
\begin{equation} 
\begin{split}
\Vert u \Vert_{C^{\eta}(B_1)}\leq C\\
\Vert v \Vert_{C^{\eta}(B_1)}\leq C
\end{split}
\end{equation} 
for some constant $C>0$ that depends only on $n$, $\lambda$, $\Lambda$, $\Vert f \Vert_{L^\infty(B_1)}$, $\Vert g \Vert_{L^\infty(B_1)}$, $\Vert u_0 \Vert_{C^\gamma(B_1)}$ and $\Vert v_0 \Vert_{C^\gamma(B_1)}$ (this is the constant $C$ used to define $X$). In particular, this implies
that $(u,v)\in X$.

\textbf{Step 3: $T$ is continuous.}

Let $T(\bar{u}, \bar{v})=(u',v')$ and $T(\bar{\bar{u}}, \bar{\bar{v}})=(u'',v'')$. We want to show that given $\epsilon>0$ we can find $\delta>0$ so that 
\[
\Vert \bar{u} - \bar{\bar{u}} \Vert_{C^{\tilde{\alpha}}(B_1)}<\delta \quad\quad\textrm{ and }\quad\quad \ \Vert \bar{v} - \bar{\bar{v}} \Vert_{C^{\tilde{\alpha}}(B_1)} < \delta
\] 
imply
\[
\Vert u'-u''\Vert_{C^{\tilde{\alpha}}(B_1)}< \epsilon \quad\quad\textrm{ and }\quad\quad \Vert v'-v''\Vert _{C^{\tilde{\alpha}}(B_1)} < \epsilon.
\] 

Notice that 
$$F(D^2u') - F(D^2u'') = \beta(\bar{u}-\bar{v}) - \beta(\bar{\bar{u}}-\bar{\bar{v}}) $$
and from the definition of uniform ellipticity we have
$$ \mathcal{M}^-(D^2(u'-u'')) \leq  F(D^2u') - F(D^2u'') \leq \mathcal{M}^+(D^2(u'-u''))$$ 
in the viscosity sense. Now let $w := u'- u''$. From the previous two inequalities we have 
\begin{equation}\label{eq.TempEqu1}
  \left\{ \begin{array}{rcll}
  \mathcal{M}^+(D^2w) & \geq &  \ \ -\Vert \beta \Vert_{C^1(\R)}(|\bar{u} - \bar{\bar{u}}|+|\bar{v} - \bar{\bar{v}}|)     & \textrm{ in } B_1 \\
  \mathcal{M}^-(D^2w) & \leq & \ \ \ \ \Vert \beta \Vert_{C^1(\R)}(|\bar{u} - \bar{\bar{u}}|+|\bar{v} - \bar{\bar{v}}|) & \textrm{ in } B_1 \\
  w & = & 0  & \textrm{ on } \partial B_1, 
\end{array}\right.
\end{equation}
so if we pick $\delta$ small enough so that $\Vert \beta \Vert_{C^1(B_1)}(|\bar{u} - \bar{\bar{u}}|+ |\bar{v} - \bar{\bar{v}}|)\leq \delta_0$  for some $\delta_0>0$ to be chosen, we can rewrite \eqref{eq.TempEqu1} as
\begin{equation}
  \left\{ \begin{array}{rcll}
  \mathcal{M}^+(D^2w) & \geq & -\delta_0    & \textrm{ in } B_1 \\
  \mathcal{M}^-(D^2w) & \leq & \ \ \delta_0    & \textrm{ in } B_1 \\
  w & = & 0  & \textrm{ in } \partial B_1.\\ 
\end{array}\right.
\end{equation}

Then, by Proposition \ref{prop.globalestbmc} and the observation following it (and choosing $\alpha=\delta_0$ in \eqref{eq.diri2}) we can pick $\delta_0$ small enough to get 
\[
\Vert u'-u''\Vert_{C^{\tilde{\alpha}}(B_1)} = \Vert w \Vert_{C^{\tilde{\alpha}}(B_1)} \leq \Vert w \Vert_{C^\eta(B_1)}\leq \epsilon
\] 
as desired. The proof of $\Vert v'-v''\Vert_{C^{\tilde{\alpha}}(B_1)} \leq  \epsilon$  \ follows in an analog way.

\end{proof}

We now show the main result of this section, which is the existence of solutions of the two membranes problem \eqref{eq.main}.

\begin{thm}
There exist $u,v\in C(B_1)$ that solve \eqref{eq.main} in the viscosity sense.
\end{thm}

\begin{proof}
Let $N$ and $\beta:\R\rightarrow\R$ be as in \eqref{eq.ene} and \eqref{eq.beta} and define $\beta_{\epsilon}(t)$ as in \eqref{eq.betaepsilon}. Now, let $u_\epsilon, v_\epsilon$ solutions of Problem \eqref{eq.penalepsilon}, i.e. $u_\epsilon, v_\epsilon$
satisfy
\begin{equation*}
  \left\{ \begin{array}{rcll}
  F(D^2u_\epsilon) & = & f+\beta_\epsilon(u_\epsilon-v_\epsilon)& \textrm{ in } B_1 \\
  G(D^2v_\epsilon) & = & g-\beta_\epsilon(u_\epsilon-v_\epsilon)& \textrm{ in } B_1 \\
  u_\epsilon & = & u_0  & \textrm{ in } \partial B_1\\
  v_\epsilon & = & v_0 & \textrm{ in } \partial B_1
\end{array}\right.
\end{equation*}

By Proposition \ref{prop.existpenal} such $u_\epsilon, v_\epsilon$ and exist. Moreover, notice that as 
\[
\Vert \beta_\epsilon \Vert_{L^\infty(\R)} = \Vert \beta \Vert_{L^\infty(\R)}\leq N 
\]
and $f, g \in L^\infty(B_1)$, Proposition \ref{prop.globalest} gives us $\Vert u_\epsilon \Vert_{C^{\eta}(B_1)},\Vert v_\epsilon \Vert_{C^{\eta}(B_1)} \leq C$ \ for some $C>0$ that does not depend on $\epsilon$. Hence, by Arsela-Ascoli (up to subsequences) 
$u_\epsilon \rightarrow u$ and $v_\epsilon \rightarrow v$  uniformly on $B_1$ for some $u,v\in C^{\tilde{\eta}}(B_1)$ where $\tilde{\eta}<\eta$. We claim that $u$ and $v$ solve \eqref{eq.main}. 

We first want to see that $u \geq v$. Assume not, i.e. suppose the exists $x\in B_1$ such that 
\[
u(x)-v(x) = -\delta <0.
\] 
From the uniform convergence we will have $u_\epsilon(x)-v_\epsilon(x) < -\delta/2$ for $\epsilon$ small enough. In particular, $u_\epsilon-v_\epsilon$ has to have a negative minimum at some point $y\in B_1$ (recall that on $\partial B_1$ we have $u_0\geq v_0$). Moreover, $u_\epsilon-v_\epsilon$
satisfies, by the convexity of $F$ and the fact that $F(X)=-G(-X)$,
\[
F\Big(D^2\Big(\frac{u_\epsilon-v_\epsilon}{2}\Big)\Big)\leq \frac{1}{2}(f-g + 2\beta_\epsilon(u_\epsilon-v_\epsilon))
\] 
in the viscosity sense. Let $P$ be a plane touching $\frac{u_\epsilon-v_\epsilon}{2}$ by below at $y$. Then  
\[
F(D^2P)\equiv 0
\] 
but since in particular $P$ is a $C^2$ function we must have 
\[
F(D^2P)\leq\frac{1}{2}(f(y)-g(y)+2\beta_\epsilon(P(y)))<0
\] 
a contradiction.

Now we want to show that $u$ and $v$ satisfy the corresponding equations. Let us start by showing that $F(D^2u)\leq f \textrm{ in } B_1$. 

Let $\varphi$ be a paraboloid touching $u$ by below at $x_0\in B_1$. Given $\xi>0$ there exists $\delta>0$ such that 
\[
f(x)\leq f(x_0)+\xi
\]
for any $x\in B_\delta(x_0)\subset B_1$. Also, for any $\eta>0$ we can choose $\epsilon$ small enough so that a translation of $\varphi(x)-\frac{\eta}{2}|x-x_0|^2$ (which we call $\tilde{\varphi}$) touches $u_\epsilon$ by below at $x_1\in B_\delta(x_0)$. Hence
\[
F(D^2\tilde{\varphi}(x_1))\leq f(x_1)+\beta_\epsilon(u_\epsilon(x_1)-v_\epsilon(x_1))\leq f(x_1)\leq f(x_0)+\xi.
\]
Since $\xi$ was arbitrary we get
\[
F(D^2\tilde{\varphi}(x_1))\leq f(x_0)
\]
but 
\[
F(D^2\tilde{\varphi}(x_1))=F(D^2\varphi(x_1)-\eta\textrm{Id})=F(D^2\varphi(x_0)-\eta\textrm{Id})
\]
and letting $\eta\rightarrow0$ we get
\[
F(D^2\varphi(x_0))\leq f(x_0)
\]
as desired (recall that $F$ is continuous in the space of matrices).

Using again the uniform convergence, the definition of viscosity solutions and considering a test function $\varphi$ that touches $u$ by above we can similarly show that  $F(D^2u)  \geq  f  \textrm{ in } B_1\cap\Omega$ and conclude that $F(D^2u)  =  f  \textrm{ in } B_1\cap\Omega$. The 
proofs of  $G(D^2v)  \geq  g \textrm{ in } B_1$ and $G(D^2v)  =  g  \textrm{ in } B_1\cap\Omega$ are analogous to the previous reasoning. It is immediate from uniform convergence that $u=u_0$ and $v=v_0$ on $\partial B_1$. 
\end{proof}

\begin{rem}
It is noting that the proof would still hold if we slightly relax the assumptions on the operators by asking just $F(X)\leq -G(-X)$.
\end{rem} 

\begin{rem}
We point out that there is no uniqueness in this problem. This comes from the fact that ``there is no equation'' on the contact set. In fact, let (for $n=1$) 
\begin{equation*}
\begin{aligned}[c]
u(x)=\left\{\begin{array}{rcll}
  \frac{x_+^2}{2} & \textrm{ for } & 0< x\leq 1 \\
  0         & \textrm{ for } & -1\leq x\leq 0 
\end{array}\right.
\end{aligned}
\qquad
\begin{aligned}[c]
v(x)=\left\{\begin{array}{rcll}
  -\frac{x_+^2}{2} & \textrm{ for } & 0< x\leq 1 \\
  0         & \textrm{ for } & -1\leq x\leq 0. 
\end{array}\right.
\end{aligned}
\end{equation*}

$u$ and $v$ are $C^{1,1}$ functions and they are strong solutions (and hence viscosity solutions) of \eqref{eq.main} with $F=\mathcal{M}^+,G=\mathcal{M}^-,f=\Lambda$ and $g=-\Lambda$. However, we can make a perturbation $\psi\in C^\infty_c((-1,0))$ such that 
\[
-1\leq \psi'' \leq 1 
\]
in $(-1,0)$ and get another solution. 

Of course this example can be easily generalized to $n\geq 2$ choosing
\begin{equation*}
\begin{aligned}[c]
u(x)=\left\{\begin{array}{rcll}
   (|x|-1/2)^2_+ & \textrm{ in } & B_1\setminus B_{1/2} \\
  0         & \textrm{ in } & B_{1/2} 
\end{array}\right.
\end{aligned}
\qquad
\begin{aligned}[c]
v(x)=\left\{\begin{array}{rcll}
   -(|x|-1/2)^2_+ & \textrm{ in } & B_1\setminus B_{1/2} \\
  0         & \textrm{ in } & B_{1/2} 
\end{array}\right.
\end{aligned}
\end{equation*}
and modifying the right hand sides accordingly. 

However, uniqueness does hold in the ``nonexcercise region'' $\Omega$. In fact, if two pairs of solutions $(u,v)$ and $(u',v')$ satisfy 
\[
 u\geq u'\quad\textrm{ and }\quad v\geq v'\quad\textrm{ on } \partial B_1\cup\partial\Omega
\]
then 
\[
 u\geq u'\quad\textrm{ and }\quad v\geq v'\quad\textrm{ in }  B_1\cap\Omega
\]
by the maximum principle for fully nonlinear elliptic equations (notice that in $B_1\cap\Omega$ we have $F(D^2u)=f$ and $G(D^2v)=g$). 
\end{rem}

\section{Regularity for the solution pair}
\label{sec.3}

In this section we prove our main regularity result, Theorem \ref{thm.main}. The fact that this is the optimal regularity can be easily seen by considering  $u(x)=\frac{x_+^2}{2}$ and $v(x)=-\frac{x_+^2}{2}$ (in one dimension) and noticing that 
they solve \eqref{eq.main} with $f\equiv\Lambda$ and $g\equiv-\Lambda$ in $[-1,1]$.

To prove Theorem \ref{thm.main} we show that solutions to \eqref{eq.main} satisfy the hypothesis of  the following Theorem (see Theorem 2.1 in \cite{IM}): 

\begin{thm}\label{thm.indreimine}
 Let $f\in C^\gamma(B_1)$ and $u$ a $W^{2,n}(B_1)$ solution of  
\begin{equation*}
  \left\{ \begin{array}{rcll}
  F(D^2u) & = & f(x) & \textrm{ in } B_1\cap\Omega \\
  |D^2u|   & \leq & C & \textrm{a.e. in } B_1\cap\Omega^c
\end{array}\right.
\end{equation*}
for some open set $\Omega\subset B_1$ and some elliptic operator $F$ that is either concave or convex. Then, there exists a constant $C$ depending only on $\|f\|_{C^\gamma(B_1)},\|u\|_{W^{2,n}(B_1)}$, the dimension and the ellipticity constants
such that 
\[
 |D^2u|\leq C \quad \textrm{a.e. in } B_{1/2}.
\]
\end{thm}

The first step is to show the following Calder\'on-Zygmund type estimate:

\begin{prop}\label{prop.cz}
Let $u$ and $v$ solve \eqref{eq.main}. Then $u$ and $v$ belong to $W^{2,p}(B_{1/2})$ for any $1<p<\infty$ and  
\begin{equation}\label{eq.czbound}
 \|u\|_{W^{2,p}(B_{1/2})}, \|v\|_{W^{2,p}(B_{1/2})} \leq C 
\end{equation}
for some constant $C$ depending only on $n,\lambda,\Lambda, \|u\|_{L^\infty(B_1)}, \|v\|_{L^\infty(B_1)}, \|f\|_{L^\infty(B_1)}$ and $\|g\|_{L^\infty(B_1)}$. 
\end{prop}

\begin{proof}
We prove the result for $u$, the proof for $v$ is analogous. We will show that $|F(D^2u)|\leq C$ in the viscosity sense for some universal constant $C$. The result will then follow from Theorem 7.1 in
in \cite{CC} (recall that $F(\cdot)$ is a convex operator). 

Let $\varphi$ be a $C^2$ function touching $u$ by below at $x_0\in B_1$. Recall that $u$ is a viscosity supersolution across the whole ball (disregarding if $x_0$ is in the contact set or not), so we have  

\[
F(D^2\varphi(x_0))\leq f(x_0)\leq\|f\|_{L^\infty(B_1)}. 
\]

If instead $\varphi$ touches $u$ by above, we separate two cases:

\textbf{Case 1:} if $x_0\in\Omega$ then $u$ is also a subsolution and we get 

\[
F(D^2\varphi(x_0))\geq f(x_0)\geq-\|f\|_{L^\infty(B_1)}. 
\]

\textbf{Case 2:} if $x_0\notin\Omega$, notice that $\varphi$ also touches $v$ by above, and $v$ is a subsolution for $G$ across the whole ball. Then 

\[
G(D^2\varphi(x_0))\geq g(x_0)\geq-\|g\|_{L^\infty(B_1)}. 
\]

But for any symmetric matrix $X$ we have $G(X)\leq F(X)$. Thus 
\[
F(D^2\varphi(x_0))\geq g(x_0)\geq-\|g\|_{L^\infty(B_1)}
\]
and we are done.
\end{proof}

\begin{rem}
Notice that the proof is still valid if we just require $F(X)\geq G(X)$.
\end{rem} 

Now we show that when problem is given by the Pucci extremal operators solutions are $C^{1,1}$ on the contact set (i.e. they have bounded second derivatives). More precisely:

\begin{prop}\label{prop.contactset}
Let $u$ and $v$ solve \eqref{eq.main} with $F=\mathcal{M}^+$ and $G=\mathcal{M}^-$. Then $u$ and $v$ are $C^{1,1}$ in $B_{1/2}\cap\Omega^c$ and  
\[
 \|D^2u\|_{L^\infty(B_{1/2}\cap\Omega^c)}, \|D^2v\|_{L^\infty(B_{1/2}\cap\Omega^c)}\leq C
\]
for some universal constant $C$.
\end{prop}

\begin{proof}
By Proposition \ref{prop.cz} $u$ and $v$ are $W^{2,p}$ functions in, say, $B_{3/4}$ so we only need to show an almost everywhere bound on $D^2u$ and $D^2v$ in $\Omega^c$. Also, since the notions
of viscosity solution and strong solution coincide for $W^{2,p}$ with $p\geq n$ (see \cite{CCKS}) we have that \eqref{eq.main} is satisfied a.e.

If $x$ is a point in the interior of $\Omega^c$ for which \eqref{eq.main} is satisfied, $u$ and $v$ coincide in a neighborhood of $x$ and hence, letting $e_u$ and $e_v$ denote the eigenvalues of $D^2u$ and $D^2v$ respectively, we find 
\begin{eqnarray*}
C\geq (f-g)(x) & \geq &  F(D^2u(x))-G(D^2v(x))\\
	       &   =  &  \Lambda\sum_{e_u>0}e_u+\lambda\sum_{e_u\leq 0}e_u-\lambda\sum_{e_v>0}e_v-\Lambda\sum_{e_v\leq 0}e_v \\
               &   =  &  (\Lambda-\lambda)\sum_{e_v}|e_v|
\end{eqnarray*}
and the result follows in this case. 

If $x\in\partial\Omega^c$ (again, a point at which \eqref{eq.main} holds), $u-v$ has a minimum at $x$ and hence $D^2(u-v)(x)$ is nonnegative definite, which in particular implies $\partial_{ee}u(x)\geq\partial_{ee}v(x)$ for any direction $e\in S^{n-1}$. 
Let us now pick a system of coordinates, say $\{e_1,\ldots,e_n\}$, in which $D^2v(x)$ is diagonal. Moreover let us assume without loss of generality that the first $m$ eigenvalues of $D^2v$ are nonpositive and the remainig
$n-m$ positive. Let then $A$ be a diagonal matrix with $\lambda$ in the first $m$ positions of its diagonal and $\Lambda$ otherwise. Since $A$ is a competitor in the $\sup$ and $\inf$ that define $F$ and $G$ respectively we have, using the equation, 
\begin{eqnarray*}
C\geq (f-g)(x) & \geq &  F(D^2u(x))-G(D^2v(x)) \geq \textrm{tr}(AD^2u(x))- \textrm{tr}(AD^2v(x))\\
	       & = &  \lambda \sum_{i=1}^{m} u_{e_ie_i} + \Lambda \sum_{i=m+1}^{n} u_{e_ie_i}-\Lambda\sum_{i=1}^{m} v_{e_ie_i}-\lambda\sum_{i=m+1}^n v_{e_ie_i} \\
	       & \geq &  \lambda \sum_{i=1}^{m} v_{e_ie_i} + \Lambda \sum_{i=m+1}^{n} v_{e_ie_i}- \Lambda \sum_{i=1}^{m} v_{e_ie_i} - \lambda \sum_{i=m+1}^n v_{e_ie_i}\\
	       & = &  (\lambda-\Lambda) \sum_{e_v\leq 0} e_v + (\Lambda-\lambda) \sum_{e_v>0}e_v\\
	       & = &  (\Lambda-\lambda)\sum_{e_v}|e_v|
\end{eqnarray*}
so we get the bound for $D^2v(x)$. The proof of the bounds for $D^2u(x)$ is completely analogous. 
\end{proof}

Finally, we can give the 
\begin{proof}[Proof of Theorem \ref{thm.main}]
Again, we prove the result for $u$. Notice that, due to Proposition \ref{prop.cz}, $u$ is $W^{2,n}(B_{3/4})$. Moreover, by Proposition \ref{prop.contactset} the Hessian of $u$ is bounded a.e. inside the contact set in $B_{1/2}$, and hence we have 
\begin{equation*}
  \left\{ \begin{array}{rcll}
  F(D^2u) & = & f(x) & \textrm{ in } B_{1/2}\cap\Omega \\
  |D^2u|   & \leq & C & \textrm{ in } B_{1/2}\cap\Omega^c
\end{array}\right.
\end{equation*}
and we can apply Theorem \ref{thm.indreimine} to get that $u\in C^{1,1}(B_{1/4})$ as desired.
\end{proof}

\section{Free boundary}
\label{sec.4}

The classic approach to study the regularity of the free boundary of the double membrane problem consists on substracting the two membranes (solutions), say $w:=u-v$, and reduce 
the situation to an obstacle-type problem (note that $w$ thus defined is nonnegative). One of the key steps of the analysis of the free boundary is to show that $w$ satisfies a 
non-degeneracy property around free boundary points; that is, given $x_0\in \partial\{ w>0\}$ one wants to show that 
\begin{equation} \label{nondegeneracy}
\sup _{\partial B_r(x_0)}w \geq C r^2 \textrm{ \  for \ } r>0
\end{equation}     
where $C>0$ is a universal constant. 

In the case of \eqref{eq.main} this property is not satisfied. Indeed, let $C$ be any positive constant and consider
\[
u(x,y):=x^2-y^2+Cx_+^3\quad\textrm{ and }\quad v(x,y):=x^2-y^2.
\]
Here $x_+=\max\{x,0\}$. Notice that 
\begin{equation*}
\left\{ \begin{array}{rcll}
  u & \geq  & v  & \textrm{ in }  B_1\\ 
  \mathcal{M}^+(D^2u) & = &  2(\Lambda - \lambda) + 6C\Lambda x_+    & \textrm{ in } B_1 \\
  \mathcal{M}^-(D^2v) & = & -2(\Lambda - \lambda)   & \textrm{ in } B_1 \\
\end{array}\right.
\end{equation*}
In particular $u,v$ solve \eqref{eq.main}, $0\in \partial\{ w>0\} $ and
\begin{equation} 
\sup _{\partial B_r(0)}w = C \sup _{\partial B_r(0)}x_+^3= Cr^3 < Cr^2
\end{equation}   
for any $r<1$.

In fact, by previous the following example we can see that no free boundary regularity can hold in general. If we make
\[
u(x,y):=x^2-y^2+\psi(x,y)\quad\textrm{ and }\quad v(x,y):=x^2-y^2
\]
with $\psi$ a nonnegative smooth function we can make the contact set arbitrarily bad and still get solutions of \eqref{eq.main}.

{\bf Acknowledgments}

The authors would like to thank the anonymous referees for their careful and detailed reading of this manuscript and for their suggestions and corrections. Also, Professor Sibru's comments on the probabilistic aspect of this problem were very enlightening. 

The authors were partially supported by NSF grant DMS-1540162. The second author was also partially supported by Colciencias and the third author by a Conicet scholarship.

\end{document}